\documentclass[12pt, reqno]{amsart}
\usepackage{amsmath, amsthm, amscd, amsfonts, amssymb, graphicx, color}
\usepackage[bookmarksnumbered, colorlinks, plainpages]{hyperref}

\makeatletter \oddsidemargin.9375in \evensidemargin \oddsidemargin
\def\Speaker{$^{*}$\protect\footnotetext{$^{*}$ S\lowercase{peaker.}}}
\def\authorsaddresses#1{\dedicatory{#1}}
\newtheorem{theorem}{Theorem}[section]

\theoremstyle{definition}
\newtheorem{definition}[theorem]{Definition}

\theoremstyle{remark}

\numberwithin{equation}{section}

\begin{document}
\setcounter{page}{1}

\noindent {\footnotesize The Extended Abstracts of \\
The 4$^{\rm th}$ Seminar on Functional Analysis and its Applications\\
2-3rd March 2016, Ferdowsi University of Mashhad, Iran}\\[1.00in]

\title[On Multipliers of Reproducing Kernel Banach and Hilbert Spaces]{On Multipliers of Reproducing Kernel Banach and Hilbert Spaces}

\author[Ebadian, Hashemi Sababe, Zallaghi]{Ali Ebadian$^1$,\Speaker Saeed Hashemi Sababe$^2$ and Maysam Zallaghi$^3$}

\authorsaddresses{$^{1,2}$ Department of Mathematics, Payame Noor University (PNU), Iran;\\
Ebadian.ali@gmail.com, Hashemi\underline{~}1365@yahoo.com\\
\vspace{0.5cm} $^3$ Department of Mathematics, University of Isfahan, Isfahan, Iran;\\ Maysam6543@gmail.com}


\begin{abstract}
This paper is devoted to the study of reproducing
kernel Hilbert spaces. We focus on multipliers of reproducing kernel Banach and Hilbert spaces. In particular we tried to extend this concept and prove some theorems.
\end{abstract}

\maketitle


\section{Introduction}
In functional analysis, a reproducing kernel Hilbert space (RKHS) is a Hilbert space associated with a kernel that reproduces every function in the space or, equivalently, where every evaluation functional is bounded. The reproducing kernel was first introduced in the 1907 work of Stanisław Zaremba concerning boundary value problems for harmonic and biharmonic functions. James Mercer simultaneously examined functions which satisfy the reproducing property in the theory of integral equations. These spaces have wide applications, including complex analysis, harmonic analysis, and quantum mechanics. Reproducing kernel Hilbert spaces are particularly important in the field of statistical learning theory because of the celebrated Representer theorem which states that every function in an RKHS can be written as a linear combination of the kernel function evaluated at the training points. More details can be found in \cite{1,2,3}.  \par
Given a set $X$, if we equip the set of all functions from $X$ to $\mathbb{F}$, $\mathcal{F}(X, \mathbb{F})$ with the usual operations of addition, $(f + g)(x) = f(x) + g(x)$, and scalar
multiplication, $(\lambda .  f)(x) = \lambda . (f(x))$, then $\mathcal{F}(X, \mathbb{F})$ is a vector space over $\mathbb{F}$. \\
Given a set $X$, we will say that $\mathcal{H}$ is a reproducing kernel Hilbert space(RKHS) on $X$ over $\mathbb{F}$, provided that:
\begin{enumerate}
\item $\mathcal{H}$ is a vector subspace of $\mathcal{F}(X, \mathbb{F})$,
\item $\mathcal{H}$ is endowed with an inner product, $\langle .,. \rangle$, making it into a Hilbert space,
\item for every $y \in X$, the linear evaluation functional, $E_y : \mathcal{H} \rightarrow \mathbb{F}$, defined by $E_y(f) = f(y)$, is bounded.
\end{enumerate}
If $\mathcal{H}$ is a RKHS on $X$, then since every bounded linear functional is given by the inner product with a unique vector in $\mathcal{H}$, we have that for every $y \in X$, there exists a unique vector, $k_y \in \mathcal{H}$, such that
\begin{equation}
 f(y) = \langle f, k_y\rangle \qquad \forall f \in  \mathcal{H} \\
\end{equation}
The function $k_y$ is called the reproducing kernel for the point $y$. The 2-variable function defined by $K(x, y) = k_y(x)$ is called the reproducing kernel for $\mathcal{H}$. \\
Note that we have,
\begin{eqnarray}
&K(x, y) = k_y(x) = \langle k_y, k_x \rangle \\
&\Vert E_y \Vert^2 = \Vert k_y \Vert^2 = \langle k_y, k_y \rangle = K(y, y).
\end{eqnarray}
\begin{definition}
Let $\mathcal{H}$ be a RKHS on $X$ with kernel function, $K$. A function
$f : X \rightarrow \mathbb{C}$ is called a multiplier of $\mathcal{H}$ provided that $f\mathcal{H} = \{fh : h \in \mathcal{H}\} \subseteq \mathcal{H}$. We let $\mathcal{M}(\mathcal{H})$ or $\mathcal{M}(K)$ denote the set of multipliers of $\mathcal{H}$. More generally, if $\mathcal{H}_i$, $i = 1, 2$ are RKHS’s on $X$ with reproducing kernels, $K_i$, $i = 1, 2$ then a function, $f : X \rightarrow \mathbb{C}$, such that $f\mathcal{H}_1 \subseteq \mathcal{H}_2$, is
called a multiplier of $\mathcal{H}_1$ into $\mathcal{H}_2$ and we let $\mathcal{M}(\mathcal{H}_1, \mathcal{H}_2)$ denote the set of multipliers of $\mathcal{H}_1$ into $\mathcal{H}_2$, so that $\mathcal{M}(\mathcal{H},\mathcal{H}) = \mathcal{M}(\mathcal{H})$.
\end{definition}
Given a multiplier, $f \in \mathcal{M}(\mathcal{H}_1, \mathcal{H}_2)$, we let $\mathcal{M}_f : \mathcal{H}_1 \rightarrow \mathcal{H}_2$, denote the linear map, $\mathcal{M}_f(h) = fh$.
Clearly, the set of multipliers, $\mathcal{M}(\mathcal{H}_1, \mathcal{H}_2)$ is a vector space and the set of multipliers, $\mathcal{M}(\mathcal{H})$, is an algebra.
\begin{definition}
A reproducing kernel Banach space (RKBS) on $X$ is a reflexive Banach space of functions on $X$ such that its topological dual $B'$ is isometric to a Banach space of functions on $X$ and the point evaluations are continuous linear functionals on both $B$ and $B'$.
\end{definition}
In this case, There is a kernel function $K:X\times X \rightarrow \mathbb{C}$ such that
\begin{equation}
[f, K(.,x)]_{\mathcal{B}}=f(x) \qquad \forall f \in \mathcal{B} \quad \forall x\in X,
\end{equation}
and $\mathcal{B}= \overline{span}\{K(.,x) ; \quad x\in X\}$
\begin{definition}
Let $X$ be a set. We call a uniformly convex and uniformly Frechet differentiable $RKBS$ on $X$ an $s.i.p.$ reproducing kernel Banach space $(s.i.p. RKBS)$.
\end{definition}
\begin{theorem}{(Riesz representation theorem) \cite{5}}
For each $g \in B'$, there exists a unique $h \in B$ such that $g = h^*$, i.e., $g(f) = [  f, h]_B, f \in B$ and $\Vert g\Vert_{B'} = \Vert h\Vert_B$ where $[.,.]_{\mathcal{B}}$ denotes the semi-inner product on $\mathcal{B}$.
\end{theorem}
\begin{definition}{(The adjoint operator in a semi-inner product space)}
Suppose $\mathcal{B}_1$ and $\mathcal{B}_2$ are tow s.i.p. Banach spaces. The adjoint operator $T^*$ for a map $T:\mathcal{B}_1\rightarrow \mathcal{B}_2$ is defined such that the domain of $T^*$ is
\begin{equation}
D(T^*)=\{ g^* \in \mathcal{B}_2^* : g^*T\quad \text{is continuous on}\quad \mathcal{B}_1 \},
\end{equation}
and $T^*:D(T^*)\rightarrow \mathcal{B}_1^C$ is defined by $T^*g^*=g^*T$ where $\mathcal{B}_1^C$ is the space of all continuous functionals on $\mathcal{B}_1$
\end{definition}
\begin{definition}
A normed vector space $V$ of functions on $X$ satisfies the Norm Consistency
Property if for every Cauchy sequence $\{f_n : n\in \mathbb{N}\}$ in $V$,
\begin{equation}
\lim_{n\rightarrow \infty} f_n(x) = 0 \quad x \in X \Longrightarrow \lim_{n\rightarrow \infty} \Vert f_n\Vert_V = 0.
\end{equation}
Suppose $X$ be a set and $\mathcal{B}$ be a s.i.p. RKBS on $X$ with $K$ as its kernel. let
\begin{equation}
\mathcal{B}^{\sharp} =span \{K(x,.) ; \quad x\in X\}
\end{equation}
We can define a new norm as follows
\begin{equation}
\Vert g \Vert_{\mathcal{B}^{\sharp}}=\sup_{f\in \mathcal{B},f\neq 0}\dfrac{\vert [f,g]_\mathcal{B}\vert}{\Vert f \Vert_{\mathcal{B}}} \qquad g\in \mathcal{B}^{\sharp}
\end{equation}
\end{definition}
\begin{theorem}\cite{4}
The norm $\Vert . \Vert_{\mathcal{B}^{\sharp}}$ is well-defined and point evaluation functionals are continuous on $\mathcal{B}^{\sharp}$ if and only if point evaluation functionals are continuous on $\mathcal{B}$.
\end{theorem}

\section{Main Results}
\begin{theorem}
Let $X$ be a set and $\mathcal{H})$ be a reproducing kernel Hilbert space on $X$. Then function $\pi_{\mathcal{H}} : \mathcal{M}(\mathcal{H})\times \mathcal{H}\rightarrow \mathcal{H}$ with $\pi_{\mathcal{H}} (f,h)=\mathcal{M}_f(h)=fh$ is a representation.
\end{theorem}
\begin{definition}
Suppose $X$ be a set and  $\mathcal{B}$ be a s.i.p. RKBS on $X$. A function $f:X\rightarrow \mathbb{C}$ is called a multiplier of $\mathcal{B}$ provided that $f\mathcal{B}=\{fg : g\in \mathcal{B}\} \subseteq \mathcal{B}$. We let $\mathcal{M}(\mathcal{B})$ denote the set of multipliers of $\mathcal{B}$.
\end{definition}
Suppose $\mathcal{M}_{\mathcal{B}}$ be the set of multipliers of a s.i.p. RKBS. It is endowed with a semi inner product inherited of $\mathcal{B}$. So it can be embedded in an inner product space. We denote $\mathcal{H}_{\mathcal{MB}}$ a Hilbert space spanned by $\mathcal{M}_{\mathcal{B}}$.
\begin{theorem}
Let $X$ be a set and $\mathcal{B}$ be a reproducing kernel Banach space on $X$. Then function $\pi_{\mathcal{B}} : \mathcal{M}(\mathcal{B})\times \mathcal{B}\rightarrow \mathcal{H}_{\mathcal{MB}}$ with $\pi_{\mathcal{B}} (f,g)=\mathcal{M}_f(g)=fg$ is a representation.
\end{theorem}
\begin{theorem}
Let $\mathcal{B}_i$, $i = 1, 2$ be s.i.p. RKBS's on $X$ with reproducing kernels, $K_i(x, y) = k_y^i (x)$, $i = 1, 2$. If $f \in \mathcal{M}(\mathcal{B}_1, \mathcal{B}_2)$, then for every $y \in
X$, $M_f^*(k_y^2) = \overline{f(y)}k_y^1$.
\end{theorem}
\begin{proof}
For any $h \in \mathcal{B}_1$, we have that
\begin{equation}
[h, \overline{f(y)}k_y^1]_1 = f(y)h(y) = [\mathcal{M}_f(h), k_y^2]_2 = [h, \mathcal{M}_f^*(k_y^2)],
\end{equation}
and hence, $\overline{f(y)}k_y^1 = \mathcal{M}_f^*(k_y^2)$.
\end{proof}

\begin{theorem}
Suppose $\mathcal{B}$ and $\mathcal{B}^{\sharp}$ defined as above. then $\mathcal{M}_{\mathcal{B}} \cong \mathcal{M}_{\mathcal{B}^{\sharp}}$
\end{theorem}
\begin{theorem}
The space of $\mathcal{B}_0=\{g\in \mathcal{B}^{\sharp} ; \quad \Vert g \Vert_{\mathcal{B}^{\sharp}}=1\}$ is a subspace of $\mathcal{B}$ as a s.i.p. RKBS.
\end{theorem}
\bigskip
\section*{Acknowledgement} Research supported in part by Kavosh Alborz Institute
of Mathematics and Applied Sciences.

\hspace{1in}

\bibliographystyle{amsplain}

\begin{thebibliography}{5}



\bibitem{1} A. Ebadian, S. Hashemi Sababe and Sh. Najafzadeh, \textit{On generalized reproducing kernel Hilbert spaces
}, Submited in Operators and Matrices, arXiv:1504.04668.

\bibitem{2} A. Ebadian, S. Hashemi Sababe and Sh. Najafzadeh, \textit{Generalized reproducing kernel Banach spaces in machine learning}, Submited in Bulletin of the Korean Mathematical Society.

\bibitem{3} G. Lumer, \textit{Semi-inner-product spaces}, AMS journals, 1960.

\bibitem{4} Guohui Song, Haizhang Zhang and Fred J. Hickernell, \textit{Reproducing Kernel Banach Spaces with the ℓ1 Norm},arXiv:1101.4388v3

\bibitem{5} Haizhang Zhang, Yuesheng Xu and Jun Zhang \textit{Reproducing Kernel Banach Spaces for Machine Learning}, Journal of Machine Learning Research 10 (2009) 2741-2775.
\end{thebibliography}

\end{document}